\documentclass[11pt]{amsart}
\usepackage{amsthm, amsfonts, amssymb, amsmath, bbold}
\usepackage{tgtermes}
\usepackage[T1]{fontenc}

\newcommand{\C}{\mathbb{C}}

\newcommand{\F}{\mathbb{F}}

\newtheorem{thm}{Theorem}[section]
\newtheorem{lem}[thm]{Lemma}
\newtheorem{cor}[thm]{Corollary}

\theoremstyle{definition}
\newtheorem{ex}{Example}
\newtheorem{defn}{Definition}
\newtheorem*{rem}{Remark}

\begin{document}
\title{Polynomials over structured grids}

\date{\today}
\author{Bogdan Nica}
\begin{abstract} 
We study multivariate polynomials over `structured' grids. We begin by proposing an interpretation as to what it means for a finite subset of a field to be structured; we do so by means of a numerical parameter, the nullity. We then  extend several results--notably, the Combinatorial Nullstellensatz and the Coefficient Theorem--to polynomials over structured grids. The main point is that the structure of a grid allows the degree constraints on polynomials to be relaxed.
\end{abstract}

\address{\newline Department of Mathematical Sciences \newline Indiana University--Purdue University Indianapolis}
\maketitle

\section{Introduction}
Given a polynomial $f\in F[X_1, \dots, X_n]$ and a finite grid $A_1\times\dots\times A_n\subseteq F^n$, where $F$ is a field, some natural questions arise:
\begin{itemize}
\item[$\bullet$] can $f$ vanish at all the grid points, or maybe at all but one of the grid points?
\item[$\bullet$] what can be said about the number of zeroes of $f$ in the grid?
\item[$\bullet$] can $f$ be recovered from its values over the grid?
\end{itemize}
Besides their intrinsic algebraic interest, such questions can have striking applications in number theory, combinatorics, and graph theory. The polynomial method is by now an established, though occasionally elusive, technique in these subjects (cf. \cite{Tao}). 

A celebrated result concerning multivariate polynomials over finite grids is the Combinatorial Nullstellensatz. It was crystallized by Noga Alon \cite{Al}; however, premonitions of this result can be detected in earlier works by Alon and collaborators. Incidentally, \cite{Al} also offers an excellent glimpse into the power of the polynomial method. 

\begin{thm}[Combinatorial Nullstellensatz] \label{Thm: CN} Let $A_1,\dots,A_n$ be finite subsets of a field $F$. Assume that a polynomial $f\in F[X_1, \dots, X_n]$ contains a monomial $X_1^{k_1}\dots X_n^{k_n}$ with non-zero coefficient, such that $k_1<|A_1|,\dots, k_n<|A_n|$ and
\[\deg(f)=k_1+\dots+k_n .\] 
Then $f(a)\neq 0$ for some grid point $a\in A_1\times\dots\times A_n$.
\end{thm}

By the degree of a multivariate polynomial we always mean its total degree. 

Algebraic aspects of the broader theme, polynomials over finite grids, have been investigated in many recent papers \cite{BS, BCPS, Cl, GM-P, L, Sch}. We may also refer to \cite{Kou, Mic} for alternate approaches to the Combinatorial Nullstellensatz.

In this paper we take a closer look at the tension between a polynomial and the grid it is evaluated on. Typical results, such as the Combinatorial Nullstellensatz, are formulated over an arbitrary grid; the degree restrictions on the polynomial reflect this freedom. Our starting idea is that we expect more rigidity over a `structured' grid, and this should translate into weaker degree restrictions placed upon the polynomial. We give meaning to the informal idea of a `structured' grid by means of a certain parameter--the nullity. We interpret increased nullity as increased structure: an arbitrary grid has the lowest possible nullity, whereas grids with genuine arithmetical structure have high nullity. We then find that, the higher the nullity of a grid, the more relaxed is the degree constraint for the corresponding polynomial.

The following three results illustrate our perspective. They are the notable corollaries of a more general statement, the structured Combinatorial Nullstellensatz (Theorem~\ref{Thm: GCN}). 

\begin{thm}[Zero-sum grids]\label{thmZ} Let $A_1,\dots,A_n$ be zero-sum subsets of a field $F$. Assume that a polynomial $f\in F[X_1, \dots, X_n]$ contains a monomial $X_1^{k_1}\dots X_n^{k_n}$ with non-zero coefficient, such that $k_1<|A_1|,\dots, k_n<|A_n|$ and 
\[\deg(f)\leq k_1+\dots+k_n +1.\] 
Then $f(a)\neq 0$ for some grid point $a\in A_1\times\dots\times A_n$.
\end{thm} 

\begin{thm}[Multiplicative grids]\label{thmM} Let $A_1,\dots,A_n$ be subsets of a field $F$, each of which is a coset of a finite multiplicative subgroup. Assume that a polynomial $f\in F[X_1, \dots, X_n]$ contains a monomial $X_1^{k_1}\dots X_n^{k_n}$ with non-zero coefficient, such that $k_1<|A_1|,\dots, k_n<|A_n|$ and 
\[\deg(f)\leq k_1+\dots+k_n +\min\{|A_1|,\dots, |A_n|\}-1.\] 
Then $f(a)\neq 0$ for some grid point $a\in A_1\times\dots\times A_n$.
\end{thm}

\begin{thm}[Additive grids]\label{thmA} Let $F$ be a field of characteristic $p$. Let $A_1,\dots,A_n$ be subsets of $F$, each of which is a coset of a finite additive subgroup. Assume that a polynomial $f\in F[X_1, \dots, X_n]$ contains a monomial $X_1^{k_1}\dots X_n^{k_n}$ with non-zero coefficient, such that $k_1<|A_1|,\dots, k_n<|A_n|$ and 
\[\deg(f)\leq k_1+\dots+k_n +(1-p^{-1})\min\{|A_1|,\dots, |A_n|\}-1.\] 
Then $f(a)\neq 0$ for some grid point $a\in A_1\times\dots\times A_n$.
\end{thm} 

Let us clarify that, in the latter two theorems, the subsets $A_1,\dots,A_n$ need not have the same underlying subgroup. A very minor additional hypothesis on the subsets $A_1,\dots,A_n$, left out for readability's sake, is that neither one is allowed to be a singleton.

So far, we have highlighted the Combinatorial Nullstellensatz and its structured extensions. But our study of polynomials over structured grids goes well beyond this result. Among other things, we prove and apply a structured extension of the Coefficient Theorem \cite{Sch, L, KP}, see Theorem~\ref{Thm: CCT} herein.


\section{Null subsets}
We introduce a key notion for the purposes of this paper, namely the nullity of a finite subset of a field. It will turn out that several equivalent definitions are available. We start with a particularly friendly one. Informally, this definition introduces the nullity of a finite subset as the lacunarity of an associated polynomial. Let it be agreed that, throughout the paper, subsets are always understood to be non-empty.

Let $F$ be a field, and let $A\subseteq F$ be a finite subset. The \emph{characteristic polynomial} of $A$ is the polynomial $\Pi_A\in F[X]$ given by
\begin{align*}
\Pi_A(X)=\prod_{a\in A} (X-a).
\end{align*}

\begin{defn}\label{defn: null}
Let $\lambda\in \{0,\dots, |A|\}$. We say that $A$ is \emph{$\lambda$-null} if, in the characteristic polynomial $\Pi_A(X)$, the coefficients of $X^{|A|-1}, \dots, X^{|A|-\lambda}$ vanish. 
\end{defn}

Being $0$-null is a void condition, so any finite subset satisfies it. A $1$-null set is commonly known as a zero-sum set. Clearly, the condition of being $\lambda$-null gets stronger as $\lambda$ increases. 

In the next result, we collect several observations on the calculus of null sets. The straightforward arguments are left to the reader.

\begin{lem}\label{lem: nullstab}
The following hold.
\begin{itemize}
\item[(i)] Nullity is invariant under scaling: if $A$ is $\lambda$-null and $c\in F^*$, then $cA$ is $\lambda$-null.
\item[(ii)] Nullity is invariant under adjoining or removing the zero element: $A$ is $\lambda$-null if and only if $A\cup\{0\}$ is $\lambda$-null.
\item[(iii)] Nullity is preserved by disjoint unions: if $A$ and $B$ are disjoint and $\lambda$-null, then $A\cup B$ is $\lambda$-null.
\end{itemize}
\end{lem}

We interpret nullity as structure. There is a two-way correlation supporting this conceptual point. On the one hand, subsets with arithmetic structure exhibit high nullity. On the other hand, subsets with very high nullity tend to be rather constrained.

\begin{ex} 
A subset $A$ is $\lambda$-null for $\lambda=|A|$ if and only if $A=\{0\}$.
\end{ex}

\begin{ex}
Let $\F_q$ be a finite field. Then $A=\F_q$ has characteristic polynomial $\Pi_A(X)=X^q-X$, and $A=\F_q^*$ has characteristic polynomial $\Pi_A(X)=X^{q-1}-1$. Thus both $A=\F_q$ and $A=\F_q^*$ are $(q-2)$-null.
\end{ex}

\begin{ex}\label{ex: multnull}
Let $A$ be a coset of a finite multiplicative subgroup. Then $A$ is $\lambda$-null for $\lambda=|A|-1$.

Indeed, let us first assume that $A\subseteq F^*$ is a finite multiplicative subgroup. As each element of $A$ is a root of the polynomial $X^{|A|}-1$, it follows that $\Pi_A(X)=X^{|A|}-1$. Therefore $A$ is $\lambda$-null for $\lambda=|A|-1$. By scaling, this remains true if $A$ is a coset of a finite multiplicative subgroup of $F^*$.

Conversely, assume a finite subset $A$ to be $\lambda$-null for $\lambda=|A|-1$. This means that $\Pi_A(X)=X^{|A|}-c$ for some $c\in F$. The degenerate case $c=0$ corresponds to $A=\{0\}$, which is in fact $\lambda$-null for $\lambda=|A|$. Consider now the case $c\neq 0$. Since $a^{|A|}=c$ for each $a\in A$, we deduce that $A\subseteq F^*$. Furthermore, picking some $a_0\in A$, we see that $a_0^{-1}A$ is contained in $\mu_{|A|}$, the multiplicative subgroup of $F^*$ which collects the roots of unity of order $|A|$. Note that $\mu_{|A|}$ has at most $|A|$ elements. Hence, by counting, it must be that $a_0^{-1}A=\mu_{|A|}$, that is $A=a_0 \mu_{|A|}$. Therefore $A$ is a coset of a multiplicative subgroup.
\end{ex}

\begin{ex}\label{ex: ore0}
Let $F$ be a field of positive characteristic $p$. Let $A$ be a coset of a finite additive subgroup; discard the degenerate case when $A$ is a singleton, that is, a coset of the additive subgroup $\{0\}$. Then $A$ is $\lambda$-null for $\lambda=(1-p^{-1})|A|-1$.

Indeed, note first that $|A|=p^e$ for some positive integer $e$; this is due to the fact that $A$ can be viewed as an affine space over the prime subfield of $F$. The key point about the characteristic polynomial of $A$ is that it takes the form 
\[\Pi_A(X)=X^{p^e}+c_{e-1} X^{p^{e-1}}+\dots+c_1 X^p+c_0 X+c_{-1}.\]
This fact is due to Oystein Ore \cite{Ore}; see also \cite{Ar}, as well as \cite[Thm.3.57]{LN} for the finite field case.
We present another argument, which we believe to be new, in Example~\ref{ex: ore}. The form of $\Pi_A(X)$ immediately implies that $A$ is $\lambda$-null for $\lambda=p^e-(p^{e-1}+1)=(1-p^{-1})|A|-1$. 

This nullity level cannot be increased, in general, as the following example shows. Let $\F_q$ be a finite field with $q=p^{e+1}$ elements. Consider the trace map $\mathrm{Tr}: \F_q\to \F_p$, given by $\mathrm{Tr}(a)=a+a^p+\dots+a^{p^e}$. The subset $A=\{a\in \F_q: \mathrm{Tr}(a)=0\}$ is an additive subgroup of size $p^e$. It is readily seen that its characteristic polynomial is $\Pi_A(X)=X^{p^e}+X^{p^{e-1}}+\dots+X^p+X$.
\end{ex}

\begin{ex}
Let $\F_q$ be a finite field, with $q>3$. If a subset $A\subseteq \F_q$ is $\tfrac{1}{2}(q-1)$-null, then $A=\F_q$, or $A=\F_q^*$.

Indeed, consider the polynomial $f=X^{q-|A|}\Pi_A(X)$. Then $f$ is monic of degree $q$, and it is fully reducible, i.e., it has $q$ roots counted with multiplicity. By R\'edei's theorem \cite{R}, one of the following holds: (a) $f=X^q-X$, or (b) $f'=0$, or (c) $f-X^q$ has degree at least $\tfrac{1}{2}(q+1)$.

Case (c) is ruled out by the nullity hypothesis on $A$. Indeed, in the polynomial $X^{q-|A|}\Pi_A(X)$, the coefficients of $X^{q-1}, \dots, X^{q-(q-1)/2}=X^{(q+1)/2}$ vanish. In case (b), no root of $f$ is simple. As every non-zero element of $A$ is a simple root of $f=X^{q-|A|}\Pi_A(X)$, we deduce that $A=\{0\}$. In this case $f=X^q$, and $f'=0$ does hold. However, $A=\{0\}$ is merely $1$-null, whereas $\tfrac{1}{2}(q-1)>1$ by our assumption on $q$. Therefore case (b) is ruled out as well. We are left with case (a). From $X^q-X=X^{q-|A|}\Pi_A(X)$, we easily deduce that either $A=\F_q$, or $A=\F_q^*$.\end{ex}

\section{The structured Combinatorial Nullstellensatz}
The following is our first main result.

\begin{thm}\label{Thm: GCN} Let $A_1,\dots,A_n$ be $\lambda$-null finite subsets of a field $F$. Assume that a polynomial $f\in F[X_1, \dots, X_n]$ contains a monomial $X_1^{k_1}\dots X_n^{k_n}$ with non-zero coefficient, such that $k_1<|A_1|,\dots, k_n<|A_n|$ and 
\[\deg(f)\leq k_1+\dots+k_n +\lambda.\] 
Then $f(a)\neq 0$ for some grid point $a\in A_1\times\dots\times A_n$.
\end{thm} 

The usual Combinatorial Nullstellensatz, Theorem~\ref{Thm: CN}, corresponds to the case $\lambda=0$ of the above theorem. 

Theorems~\ref{thmZ}, ~\ref{thmM}, and ~\ref{thmA} from the Introduction are immediate applications of the above theorem. If each one of $A_1,\dots,A_n$ is a zero-sum subset, then they are jointly $1$-null. If each one of $A_1,\dots,A_n$ is a coset of a finite multiplicative subgroup then, by Example~\ref{ex: multnull}, they are jointly $\lambda$-null for $\lambda=\min\{|A_1|,\dots, |A_n|\}-1$. If each one of $A_1,\dots,A_n$ is a coset of a finite additive subgroup and the ambient field $F$ has characteristic $p$, then, by Example~\ref{ex: ore0}, they are jointly $\lambda$-null for $\lambda=(1-p^{-1})\min\{|A_1|,\dots, |A_n|\}-1$.

We give a proof of Theorem~\ref{Thm: GCN} by exploiting an algebraic result of Alon \cite[Thm.1.1]{Al} which is closely related to the Combinatorial Nullstellensatz. In fact, this algebraic result also goes under the name of Combinatorial Nullstellensatz, though it is much less used. In a subsequent section, we will also see Theorem~\ref{Thm: GCN} as a consequence of Theorem~\ref{Thm: CCT}.

\begin{proof}
Arguing by contradiction, let us assume that $f(a)= 0$ for all grid points $a\in A_1\times\dots\times A_n$. Then by \cite[Thm.1.1]{Al}, the following holds: there are polynomials $h_1,\dots,h_n\in F[X_1,\dots, X_n]$ so that
\begin{align}\label{eq: cna}
f=h_1\cdot \Pi_{A_1}(X_1)+\dots+h_n\cdot \Pi_{A_n}(X_n),
\end{align}
 and, furthermore, the total degree of each polynomial $h_i$ satisfies 
\begin{align}\label{eq: cnb}
\deg(h_i)\leq \deg(f)-\deg( \Pi_{A_i})=\deg(f)-|A_i|.
\end{align}
From \eqref{eq: cna}, we deduce that the monomial $X_1^{k_1}\dots X_n^{k_n}$ appears with a non-zero coefficient in some product $h_i\cdot \Pi_{A_i}(X_i)$. Now we use the hypothesis that $A_i$ is $\lambda$-null: $\Pi_{A_i}(X_i)$ has the lacunary form
\[\Pi_{A_i}(X_i)=X_i^{|A_i|}+\sum_{r=0}^{|A_i|-\lambda-1} c_r X_i^r.\]
In the product $h_i\cdot \Pi_{A_i}(X_i)$, the monomial $X_1^{k_1}\dots X_n^{k_n}$ cannot arise from $X_i^{|A_i|}$, the leading term of $\Pi_{A_i}(X_i)$, since $k_i<|A_i|$. Therefore a term of much lower order, $X_i^r$ for some $r<|A_i|-\lambda$, is involved. This means that the monomial $X_1^{k_1}\dots X_i^{k_i-r}\dots X_n^{k_n}$ appears with a non-zero coefficient in $h_i$. But then
\begin{align*}
\deg(h_i)&\geq k_1+\dots+(k_i-r)+\dots+k_n\\
&> k_1+\dots+k_n +\lambda-|A_i|\geq \deg(f)-|A_i|
\end{align*}
in contradiction with \eqref{eq: cnb}.
\end{proof}

As an illustration, we prove a structured version of the well-known Cauchy - Davenport inequality. Let us introduce a  natural notation: given a subset $A\subseteq F$, we put 
\[\lambda(A)=\max\{\lambda: A \textrm{ is } \lambda\textrm{-null}\}.\] 

\begin{thm}\label{thm: SCD} Let $\F_p$ be a finite field with $p$ elements, where $p$ is a prime. Let $A, B\subseteq \F_p$, and consider the sumset $A+B\subseteq \F_p$. Then either
\begin{align*}
\lambda(A+B)\geq \min\{\lambda(A),\lambda(B)\},
\end{align*}
or
\begin{align*}
|A+B|\geq |A|+|B|+\lambda(A+B).
\end{align*}
\end{thm}

Informally, this says that a sumset of structured sets is either fairly structured, or else not too small.

\begin{proof} We adapt the argument of \cite[Thm.3.2]{Al}. Put $C:=A+B$, and consider the polynomial
 \[f=\prod_{c\in C} (X+Y-c)\in \F_p[X,Y].\]
Then $f$ vanishes over the grid $A\times B$. 

Arguing by contradiction, let us assume that $\lambda(C)<\min\{\lambda(A),\lambda(B)\}$ and $|C|<|A|+|B|+\lambda(C)$. Set $\mu:=\lambda(C)+1$. Then $\mu\leq \min\{\lambda(A),\lambda(B)\}$ and $|C|\leq |A|+|B|+\lambda(C)-1=(|A|-1)+(|B|-1)+\mu$. Choose non-negative integers $k$ and $\ell$ satisfying $k\leq |A|-1$, $\ell\leq |B|-1$, and $|C|=k+\ell+\mu$.

We may apply Theorem~\ref{Thm: GCN} since $A$ and $B$ are $\mu$-null, and $\deg(f)=k+\ell+\mu$, where $k<|A|$ and $\ell<|B|$. As $f$ vanishes over the grid $A\times B$, it follows that the coefficient of $X^{k}Y^{\ell}$ in $f$ is zero. Now let us find, explicitly, the coefficient of $X^kY^\ell$. Suppose that the characteristic polynomial of $C$ expands as $\Pi_C(Z)=\sum_{r=0}^{|C|} e_rZ^{|C|-r}$. Then
\begin{align*} 
f=\Pi_C(X+Y)=\sum_{r=0}^{|C|} e_r(X+Y)^{|C|-r}.
\end{align*}
The monomial $X^kY^\ell$ appears in the expansion of $(X+Y)^{k+\ell}$, which corresponds to $r=\mu$. Hence, the coefficient of $X^kY^\ell$ equals
\begin{align*}
\binom{k+\ell}{k}e_\mu.
\end{align*}
Recall that $\mu=\lambda(C)+1$, which implies that $e_\mu\neq 0$. Also $k+\ell<|C|\leq p$, which implies that the above binomial coefficient is non-zero in $\F_p$. We have therefore obtained a contradiction, as desired.
\end{proof}

In principle, the key instance of Theorem~\ref{Thm: GCN} is when $k_1=|A_1|-1,\dots, k_n=|A_n|-1$. If one knows this particular case, then one can generalize to arbitrary $k_1<|A_1|,\dots, k_n<|A_n|$. The simplest way would be to trim the grid, as in \cite{Al}, which would be legitimate if we worked over arbitrary grids. Over structured grids, the germane idea is to adapt the polynomial by a degree-raising trick--namely, consider the polynomial 
\begin{align*}
X_1^{|A_1|-k_1-1}\dots X_n^{|A_n|-k_n-1} f(X_1,\dots,X_n).
\end{align*}

The general form of Theorem~\ref{Thm: GCN} is more flexible in applications, and the proof of Theorem~\ref{thm: SCD} above is a case in point. Moving forward, in Section~\ref{sec: CCT}, it will be convenient to adopt the case $k_1=|A_1|-1,\dots, k_n=|A_n|-1$ as the main case of interest; this will simplify the formulas appearing therein. However, one should keep in mind the degree-raising trick.

\section{Nullity and symmetric moments}
 We now turn to a different viewpoint on nullity. Again, let $F$ be a field and let $A\subseteq F$ be a finite subset. If $A$ has $m$ elements, and $g(X_1,\dots,X_m)$ is a symmetric polynomial in $m$ variables, then we may evaluate $g$ on $A$ by setting 
 \[g(A)=g(a_1,\dots,a_m)\]
where $a_1,\dots,a_m$ is an enumeration of $A$. This is unambiguous: the evaluation $g(A)$ depends only on the set $A$, and not on the actual enumeration of $A$, since the polynomial $g$ is symmetric.

The elementary symmetric polynomials and the complete symmetric polynomials of degree $r$ in $m$ variables are given by
\begin{align*}
e_r(X_1,\dots,X_m)=\sum_{1\leq i_1<\dots<i_r\leq m} X_{i_1}\dots X_{i_r},\\
h_r(X_1,\dots,X_m)=\sum_{k_1+\dots+k_m=r} X_1^{k_1}\dots X_m^{k_m}.
\end{align*}
By convention, $e_0(X_1,\dots,X_m)=1$ and $h_0(X_1,\dots,X_m)=1$. Note, in addition, that $e_r(X_1,\dots,X_m)=0$ when $r>m$.

By evaluating the elementary symmetric polynomials and the complete symmetric polynomials on $A$, one obtains the \emph{elementary moments} $e_r(A)$, respectively the \emph{complete moments} $h_r(A)$. We may give an alternate description of the nullity of $A$, in terms of these symmetric moments.

\begin{lem}\label{lem: null}
Let $\lambda\in \{0,\dots,|A|\}$.  Then the following are equivalent:
\begin{itemize}
\item[(i)] $A$ is $\lambda$-null;
\item[(ii)] the elementary moments of $A$ vanish up to degree $\lambda$, that is, $e_r(A)=0$ for all $1\leq r\leq \lambda$;
\item[(iii)] the complete moments of $A$ vanish up to degree $\lambda$, that is, $h_r(A)=0$ for all $1\leq r\leq \lambda$.
\end{itemize}
\end{lem}

The zeroth moments cannot partake in vanishing, as $e_0(A)=h_0(A)=1$.

\begin{proof}
The equivalence of (i) and (ii) is immediate from Vi\`ete's formula:
\begin{align}\label{eq: viete}
\Pi_A(X)=\prod_{a\in A} (X-a)=\sum_{i=0}^{|A|} (-1)^i e_{i}(A)\: X^{|A|-i}.
\end{align}

The elementary and the complete symmetric polynomials are entwined by the identity
\begin{align*}
\sum_{i=0}^r (-1)^i e_{r-i}(X_1,\dots,X_m)\: h_{i}(X_1,\dots,X_m)=0.
\end{align*}
Evaluating on a finite subset $A$ gives
\begin{align}\label{eq: entwine}
\sum_{i=0}^r (-1)^i e_{r-i}(A)\: h_{i}(A)=0.
\end{align}
This relation bridging the two types of moments implies, in particular, the equivalence of (ii) and (iii). The argument uses an obvious induction.
\end{proof}

As the proof shows, it is evident that the nullity of $A$ is reflected in the vanishing of elementary moments of $A$. The purpose of the above lemma is to uncover the less evident link between the nullity of $A$ and the vanishing of  complete moments of $A$. This link will be of key importance for an up-coming result, Theorem~\ref{Thm: CCT}.

\begin{ex}\label{ex: ore}
Let $F$ be a field of characteristic $p$, where $p$ is a prime, and let $A\subseteq F$ be a coset of a finite additive subgroup. We show that $e_r(A)=0$ for $1\leq r\leq |A|-1$, except possibly when $|A|-r$ is a power of $p$. In view of \eqref{eq: viete}, this precisely translates into the statement, made in Example~\ref{ex: ore0}, that the characteristic polynomial $\Pi_A(X)$ is a linear combination of monomials whose degree is a power of $p$, or $0$. 

Let $k\leq |A|-1$. Consider the polynomial 
\begin{align*}
\Delta_k(X)=e_k(X+A)-e_k(A)
\end{align*} 
where $X+A$ is the translate of the subset $A$ by the indeterminate $X$. Then $\Delta_k(X)$ has degree at most $k$, and it admits $|A|$ roots--namely, the elements of the finite additive subgroup underlying $A$. Therefore $\Delta_k(X)$ is the zero polynomial. Let us put $\Delta_k(X)$ in standard form, so that we can equate its coefficients to $0$. We have
\begin{align*}
e_k(X+A)&=\sum_{B\subseteq A, \: |B|=k\:} \prod_{b\in B} (X+b)=\sum_{B\subseteq A, \: |B|=k\:} \sum_{r=0}^k e_{r}(B)\: X^{k-r}.
\end{align*}
As
\begin{align*}
\sum_{B\subseteq A, \: |B|=k\:} e_{r}(B)=\binom{|A|-r}{k-r}\: e_{r}(A),
\end{align*}
we deduce that
\begin{align*}
\Delta_k(X)=e_k(X+A)-e_k(A)=\sum_{r=0}^{k-1} \binom{|A|-r}{k-r}\: e_{r}(A)\: X^{k-r}.
\end{align*}
Thus, in $F$,
\begin{align}\label{eq: binom}
\binom{|A|-r}{k-r}\: e_{r}(A)=0, \qquad 0\leq r\leq k-1.
\end{align}

Let now $r\in \{1,\dots, |A|-1\}$, such that $|A|-r$ is not a power of $p$, be given. By a well-known property of binomial coefficients, see Fine \cite{F}, there exists $j\in \{1,\dots,|A|-r-1\}$ such that
\begin{align*}
\binom{|A|-r}{j} \not\equiv 0 \mod p.
\end{align*}
Set $k=j+r$. Note that $r< k\leq |A|-1$, so we are in position to invoke \eqref{eq: binom}. We infer that $e_{r}(A)=0$, as desired.

Furthermore, we claim that $e_{r}(A)\neq 0$ when $r=|A|-1$. We have $e_r(A)=e_r(c+A)$ for all $c\in F$ since $\Delta_r(X)$ is the zero polynomial. By definition, $A$ is a translate of a finite additive subgroup $A'$. Thus $e_r(A)=e_r(A')$. The advantage in replacing the coset $A$ by its underlying subgroup $A'$, is that we can easily compute
\begin{align*}
e_r(A')=\sum_{B\subseteq A', \: |B|=r\:} \prod_{b\in B} b=\prod_{b\in A'\setminus\{0\}} b
\end{align*}
since all but one of the indexing subsets $B$ contain $0$. The latter product is evidently non-zero.

This computation reveals the coefficient of $X$ in the characteristic polynomial $\Pi_A(X)$. Indeed, by \eqref{eq: viete}, the coefficient equals $(-1)^{|A|-1}\:e_{|A|-1}(A)$. Observe that $(-1)^{|A|-1}=1$ in $F$; this is clearly true if $p=2$, whereas for $p>2$ we recall that $|A|$ is odd, being a power of $p$. In summary, the coefficient of $X$ in $\Pi_A(X)$ is 
\[e_{|A|-1}(A)=\prod_{b\in A'\setminus\{0\}} b\]
where, once again, $A'$ is the additive subgroup underlying the coset $A$. 
\end{ex}

\section{Vandermonde subsets}

There is yet another fundamental family of symmetric polynomials, the power-sum polynomials
\begin{align*}
p_r(X_1,\dots,X_m)=X_1^{r}+\dots+ X_m^{r}.
\end{align*}
By evaluating them on a finite subset $A\subseteq F$, we obtain the \emph{power-sum moments} $p_r(A)$. These moments are actually classical, whereas the consideration of their elementary and complete counterparts appears to be new. 

\begin{defn}\label{defn: Vandy}
Let $\lambda\in \{0,\dots, |A|\}$. A finite subset $A\subseteq F$ is \emph{$\lambda$-Vandermonde} if the power-sum moments of $A$ vanish up to degree $\lambda$, that is, $p_r(A)=0$ for $1\leq r\leq \lambda$.
\end{defn}

This terminology resonates with the notion of Vandermonde subset of a finite field, as studied by Peter Sziklai and Marcella Tak\'ats \cite{ST}. The Vandermonde condition shares some general features with nullity: any finite subset is $0$-Vandermonde; being $\lambda$-Vandermonde gets stronger as $\lambda$ increases; being $\lambda$-Vandermonde is stable under the operations indicated in Lemma~\ref{lem: nullstab}.

Linking the power-sum polynomials to the elementary symmetric polynomials is Newton's formula:
\begin{align*}
re_r(X_1,\dots,X_m)+\sum_{i=1}^r (-1)^i e_{r-i}(X_1,\dots,X_m)\: p_{i}(X_1,\dots,X_m)=0
\end{align*}
for $1\leq r\leq m$. Evaluating on a finite subset $A$ gives
\begin{align}\label{eq: newton}
re_r(A)+\sum_{i=1}^r (-1)^i e_{r-i}(A)\: p_{i}(A)=0, \qquad 1\leq r\leq |A|.
\end{align}

While being $\lambda$-null and being $\lambda$-Vandermonde are not equivalent, the two notions can be related by using \eqref{eq: newton}.

\begin{lem}
If $A$ is $\lambda$-null, then $A$ is $\lambda$-Vandermonde. The converse is true provided that $\mathrm{char}\:F=0$, or that $\mathrm{char}\:F=p$ and $\lambda<p$.
\end{lem}

\begin{ex}\label{ex: AddVan}
Let $F$ be a field of characteristic $p$, and let $A\subseteq F$ be a coset of a finite additive subgroup. Then $A$ is $\lambda$-Vandermonde for $\lambda=|A|-2$, though not for $\lambda=|A|-1$. 

Indeed, we know from Example~\ref{ex: ore} that $e_r(A)=0$ whenever $r\in \{1, \dots, |A|-2\}$ is not a multiple of $p$. An inductive use of \eqref{eq: newton} yields $p_r(A)=0$ for $1\leq r\leq |A|-2$. Thus $A$ is $\lambda$-Vandermonde for $\lambda=|A|-2$. 

For $r=|A|-1$, \eqref{eq: newton} gives $re_r(A)+(-1)^rp_r(A)=0$. This amounts to $p_r(A)=(-1)^re_r(A)$ since $r=-1$ in $F$. We also know from Example~\ref{ex: ore} that $e_{r}(A)\neq 0$, whence $p_{r}(A)\neq 0$. Thus $A$ is not $\lambda$-Vandermonde for $\lambda=|A|-1$.

Here is an alternate argument, which explains the `Vandermonde' terminology. We know that $p_r(A)=0$ for $1\leq r\leq |A|-2$, but also for $r=0$ since the size of $A$ is a multiple of $p$. If we had $p_r(A)= 0$ for $r=|A|-1$ as well, that would contradict the invertibility of a Vandermonde matrix associated to $A$.
\end{ex}

Vandermonde subsets are not the main focus of this paper, but they are worth mentioning on two accounts. Firstly, we believe that the power-sum moments are useful for a more systematic study of nullity--a study that we do not attempt herein. For instance, the power-sum viewpoint makes it obvious that, in formally real fields, one can only speak of $1$-nullity. This observation is not immediate from Definition~\ref{defn: null}. Secondly, Vandermonde subsets fit our main theme--the interplay between polynomials and structured grids. To wit, we have the following pair of results.

\begin{thm}\label{thm: War1}
Let $A_1,\dots,A_n$ be $\lambda$-Vandermonde finite subsets of a field $F$. Let $f\in F[X_1, \dots, X_n]$ be a polynomial with the property that the degree of each variable is at most $\lambda$. Then
\[\sum_{a\in A_1\times \dots \times A_n} f(a)=f(0) |A_1|\dots |A_n|.\] 
\end{thm}

\begin{thm}\label{thm: War2}
Let $F$ be a field of characteristic $p$. Let $A_1,\dots,A_n$ be $\lambda$-Vandermonde finite subsets of $F$, such that $p$ divides $|A_1|,\dots, |A_n|$. Let $f\in F[X_1, \dots, X_n]$ be a polynomial whose degree satisfies $\deg(f)< n(\lambda+1)$. Then
\[\sum_{a\in A_1\times \dots \times A_n} f(a)=0.\] 
\end{thm}

Theorems~\ref{thm: War1} and ~\ref{thm: War2} are closely related, and we give a joint proof below. But they work somewhat differently: in Theorem~\ref{thm: War2}, the underlying grid is a bit more structured, allowing for the degree requirement on $f$ to be relaxed.

\begin{proof}[Proof of Theorems~\ref{thm: War1} and ~\ref{thm: War2}.] 
Let $f_{d_1,\dots,d_n}$ denote the coefficient of $X_1^{d_1}\dots X_n^{d_n}$ in $f$. So 
\begin{align*}
f=\sum_{d_1,\dots,d_n} f_{d_1,\dots,d_n} \:X_1^{d_1}\dots X_n^{d_n}
\end{align*}
where each of $d_1,\dots,d_n$ runs over the non-negative integers. Then
\begin{align}\label{eq: eqWar}
\sum_{a\in A_1\times \dots \times A_n} f(a)=\sum_{d_1,\dots,d_n} f_{d_1,\dots,d_n} \Big( \sum_{a_1\in A_1}a_1^{d_1}\Big)\dots \Big( \sum_{a_n\in A_n} a_n^{d_n}\Big).
\end{align}
As $A_1,\dots,A_n$ are $\lambda$-Vandermonde, we may restrict the summation on the right-hand side of \eqref{eq: eqWar} to those tuples satisfying $d_i=0$ or $d_i> \lambda$, for each $i$. 

The degree assumption of Theorem~\ref{thm: War1} is that $f_{d_1,\dots,d_n}=0$, except possibly when $d_i\leq \lambda$ for each $i$. Thus, on the right-hand side of \eqref{eq: eqWar}, only the term corresponding to $d_1=\dots=d_n=0$ remains. We obtain 
\[\sum_{a\in A_1\times \dots \times A_n} f(a)=f(0) |A_1|\dots |A_n|.\]

Consider now the setup of Theorem~\ref{thm: War2}. The cardinality assumption on the sets $|A_1|$, $\dots$, $|A_n|$
means that the indexing on the right-hand side of \eqref{eq: eqWar} can be further restricted to those tuples satisfying $d_i> \lambda$, for each $i$. On the other hand the degree assumption on $f$ is that $f_{d_1,\dots,d_n}=0$ when $d_1+\dots+d_n \geq n(\lambda+1)$. Therefore
\[\sum_{a\in A_1\times \dots \times A_n} f(a)=0\] 
in this case.
\end{proof}

By Example~\ref{ex: AddVan} above, Theorem~\ref{thm: War2} applies when $A_1,\dots,A_n$ are finite additive subgroups and, more generally, cosets thereof. We deduce the following consequence.

\begin{cor}\label{cor: War2add}
Let $F$ be a field of characteristic $p$, and let $A_1,\dots,A_n$ be finite additive subgroups of $F$. Let $f\in F[X_1, \dots, X_n]$ be a polynomial whose degree satisfies 
\[\deg(f)< n\:\big(\min\{|A_1|,\dots,|A_n|\}-1\big).\] 
Then
\[\sum_{a\in A_1\times \dots \times A_n} f(a)=0.\] 
\end{cor}

We strengthen the above corollary in the next section. 

\begin{rem}[Pete Clark, personal communication]
Anurag Bishnoi and Pete Clark have independently obtained closely related results on polynomials over Vandermonde grids. Their unpublished work is reported in \cite{Cl2}.
\end{rem}

\section{The complete coefficient theorem}\label{sec: CCT}
For the purposes of this section, it will be convenient to introduce a notation. Let $F$ be a field. For each point $a=(a_1,\dots,a_n)$ in a finite grid $A_1\times \dots \times A_n\subseteq F^n$, put 
\begin{align*}
w_a=\frac{1}{\Pi'_{A_1}(a_1)\dots \Pi'_{A_n}(a_n)}.
\end{align*}
In this formula, $\Pi'_{A_1}, \dots, \Pi'_{A_n}$ are the formal derivatives of the characteristic polynomials of $A_1$, $\dots$, $A_n$. Given a finite subset $A\subseteq F$, we have $\Pi'_A(a)\neq 0$ for each $a\in A$ since $\Pi_A$ is, by definition, a separable polynomial. We can actually write down an explicit formula:
\begin{align*}
\Pi'_A(a)= \prod_{b\in A, b\neq a} (a-b).
\end{align*}

The following Coefficient Theorem--due, independently, to Uwe Schauz \cite{Sch}, Micha\l{} Laso\'n \cite{L}, Roman Karasev and Fedor Petrov \cite{KP}--is surprisingly recent.

\begin{thm}[\cite{Sch, L, KP}]\label{Thm: CT}  Let $A_1,\dots,A_n$ be finite subsets of a field $F$. Assume that $f\in F[X_1, \dots, X_n]$ is a polynomial whose degree satisfies
\[\deg(f)\leq (|A_1|-1)+\dots+(|A_n|-1).\] 
Then the coefficient of the monomial $X_1^{|A_1|-1}\dots X_n^{|A_n|-1}$ in $f$ equals
\begin{align*}
\sum_{a\in A_1\times \dots \times A_n} w_a\: f(a).
\end{align*}
\end{thm}

Our second main result is a generalization of Theorem~\ref{Thm: CT}, that we term the Complete Coefficient Theorem.

\begin{thm}\label{Thm: CCT}  Let $A_1,\dots,A_n$ be $\lambda$-null finite subsets of a field $F$. Assume that $f\in F[X_1, \dots, X_n]$ is a polynomial whose degree satisfies
\[\deg(f)\leq (|A_1|-1)+\dots+(|A_n|-1)+\lambda.\] 
Then the coefficient of the monomial $X_1^{|A_1|-1}\dots X_n^{|A_n|-1}$ in $f$ equals
\begin{align*}
\sum_{a\in A_1\times \dots \times A_n} w_a\: f(a).
\end{align*}
\end{thm} 

A key role in the proof of the theorem is played by the following fact.

\begin{lem}[Sylvester's identity]\label{lem: Syl}
Let $A$ be a finite subset of a field $F$, and let $d$ be a non-negative integer. Then
\begin{align*}
\sum_{a\in A} \frac{a^{d}}{\Pi'_{A}(a)}=\begin{cases}
0 & \textrm{ if } 0\leq d< |A|-1,\\
h_{d-|A|+1}(A) & \textrm{ if } d\geq |A|-1.
\end{cases}
\end{align*}
\end{lem} 

We recall that $h_{d-|A|+1}(A)$ denotes the complete moment of degree $d-|A|+1$. We refer to the Appendix for a leisurely, self-contained discussion of the lemma.

\begin{proof} Let $f_{d_1,\dots,d_n}$ denote the coefficient of the monomial $X_1^{d_1}\dots X_n^{d_n}$ in $f$. So 
\begin{align*}
f=\sum_{d_1,\dots,d_n} f_{d_1,\dots,d_n} \:X_1^{d_1}\dots X_n^{d_n}
\end{align*}
where each of $d_1,\dots,d_n$ runs over the non-negative integers. Using this monomial expansion of $f$, we can write
\begin{align*}
\sum_{a\in A_1\times \dots \times A_n} w_a\: f(a)=\sum_{d_1,\dots,d_n} f_{d_1,\dots,d_n} \bigg( \sum_{a_1\in A_1} \frac{a_1^{d_1}}{\Pi'_{A_1}(a_1)}\bigg)\dots \bigg( \sum_{a_n\in A_n} \frac{a_n^{d_n}}{\Pi'_{A_n}(a_n)}\bigg).
\end{align*}
By Sylvester's identity, the right-hand sum equals
\begin{align*}
&\sum_{d_1\geq |A_1|-1,\dots, d_n\geq |A_n|-1} f_{d_1,\dots,d_n}\cdot h_{d_1-|A_1|+1}(A_1) \dots h_{d_n-|A_n|+1}(A_n)\\
&=\sum_{r_1\geq 0,\dots, r_n\geq 0} f_{r_1+|A_1|-1,\dots, r_n+ |A_n|-1}\cdot h_{r_1}(A_1) \dots h_{r_n}(A_n).
\end{align*}

We need to show that the latter sum equals $f_{|A_1|-1,\dots, |A_n|-1}$. This is precisely the contribution of the multiindex $(r_1,\dots,r_n)=(0,\dots,0)$, as $h_0(A)=1$ for any finite subset $A$. Next, we check that the contribution of a multiindex $(r_1,\dots,r_n)\neq (0,\dots,0)$ vanishes. If $ r_1+\dots+r_n> \lambda$, then $f_{r_1+|A_1|-1,\dots, r_n+ |A_n|-1}=0$ by the degree hypothesis on $f$. If $r_1+\dots+r_n\leq \lambda$, then we argue that some $h_{r_i}(A_i)$ vanishes. Indeed, let $i$ be an index for which $r_i\neq 0$. Then $1\leq r_i\leq \lambda$. It follows from the $\lambda$-nullity of $A_i$, as interpreted through Lemma~\ref{lem: null}, that $h_{r_i}(A_i)=0$. \end{proof}

Theorem~\ref{Thm: CCT} might seem daunting, due to the complicated formula of the weight function $w$. On the one hand, here are two applications in which the weights' formula is actually irrelevant. Firstly, keep the hypotheses of Theorem~\ref{Thm: CCT} and assume further that the coefficient of the monomial $X_1^{|A_1|-1}\dots X_n^{|A_n|-1}$ in $f$ is non-zero; then $f$ cannot vanish at each point of the grid $A_1\times \dots \times A_n$. This is the key case, $k_1=|A_1|-1$, $\dots$, $k_n=|A_n|-1$, of Theorem~\ref{Thm: GCN}. Secondly, if we now assume the coefficient of the monomial $X_1^{|A_1|-1}\dots X_n^{|A_n|-1}$ to be zero, we obtain the following consequence.

\begin{cor}\label{cor: punctured} Let $A_1,\dots,A_n$ be $\lambda$-null finite subsets of a field $F$. Assume that $f\in F[X_1, \dots, X_n]$ is a polynomial whose degree satisfies
\[\deg(f)\leq (|A_1|-1)+\dots+(|A_n|-1)+\lambda.\] 
If the coefficient of the monomial $X_1^{|A_1|-1}\dots X_n^{|A_n|-1}$ in $f$ is zero, then $f$ cannot vanish at all but one point of the grid $A_1\times \dots \times A_n$.
\end{cor} 

Over finite fields, the previous corollary yields results in the spirit of Chevalley's theorem. We illustrate the idea on the following geometric example.

\begin{ex}\label{ex: PP}
Let $A_1,\dots,A_n$ be subsets of a finite field $\F_q$. The feature of the grid $A_1\times\dots \times A_n\subseteq \F_q^n$ that we are interested in is
\begin{itemize}
\item[(PP)] a plane in $\F_q^n$ cannot intersect the grid in a single point.
\end{itemize}
A plane $\pi\subseteq \F_q^n$ has an associated polynomial $f_\pi\in \F_q[X_1,\dots,X_n]$ of degree $q-1$ whose support is $\pi$. Namely, if $\pi$ is given by the equation $c_1x_1+\dots+c_nx_n=0$, then $f_\pi=1-(c_1X_1+\dots+c_nX_n)^{q-1}$ satisfies $f_\pi(x)\neq 0$ if and only if $x\in \pi$. The intersection of $\pi$ and the grid $A_1\times\dots \times A_n$ consists of those points in the grid where $f_\pi$ does not vanish. We aim to apply Corollary~\ref{cor: punctured} to polynomials of the form $f_\pi$. As $f_\pi$ is homogeneous, we can ensure that the monomial $X_1^{|A_1|-1}\dots X_n^{|A_n|-1}$ does not appear in $f_\pi$ by requiring $(|A_1|-1)+\dots+(|A_n|-1)\neq q-1$.

The unstructured outcome is the following: if the grid is large, in the sense that 
\[(|A_1|-1)+\dots+(|A_n|-1)>q-1,\]
then (PP) holds. The structured upshot of Corollary~\ref{cor: punctured} deals with smaller null grids: if $A_1,\dots,A_n$ are $\lambda$-null and
\[q-1-\lambda\leq (|A_1|-1)+\dots+(|A_n|-1)<q-1,\]
then (PP) holds.
\end{ex}

On the other hand, we can use Theorem~\ref{Thm: CCT} over structured grids whose weight function $w$ ends up having a much simpler form. Here are two key examples.

\smallskip
(i) Let $A\subseteq F^*$ be a finite multiplicative subgroup. Then the characteristic polynomial of $A$ is of the form $\Pi_A=X^{|A|}-c$, hence
\begin{align*}
\Pi'_A(a)=|A|\: a^{|A|-1}=|A|\: a^{-1}
\end{align*}
for each $a\in A$. It follows that, for a grid $A_1\times\dots \times A_n$ defined by finite multiplicative subgroups $A_1,\dots,A_n \subseteq F^*$, we have
\begin{align*}
w_a=\frac{1}{\Pi'_{A_1}(a_1)\dots \Pi'_{A_n}(a_n)}=\frac{a_1\dots a_n}{|A_1\times\dots \times A_n|}
\end{align*}
for each grid point $a=(a_1,\dots,a_n)$.

\smallskip
(ii) Let $A\subseteq F$ be a finite additive subgroup. Then 
\begin{align*}
\Pi'_A(a)= \prod_{b\in A, b\neq a} (a-b)=\prod_{b\in A\setminus\{0\}} b
\end{align*}
for each $a\in A$. Thus, for a grid $A_1\times\dots \times A_n$ defined by finite additive subgroups $A_1,\dots,A_n \subseteq F$, we have
\begin{align*}
w_a=\frac{1}{\Pi'_{A_1}(a_1)\dots \Pi'_{A_n}(a_n)}=\bigg(\prod_{b_1\in A_1\setminus\{0\}} b_1^{-1}\bigg)\dots \bigg(\prod_{b_n\in A_n\setminus\{0\}} b_n^{-1}\bigg)
\end{align*}
for each grid point $a=(a_1,\dots,a_n)$. The notable feature in this case is that the weights are constant over the grid.

\smallskip
We deduce the following high nullity instances of Theorem~\ref{Thm: CCT}.

\begin{cor}[Multiplicative grids] Let $A_1,\dots,A_n$ be finite multiplicative subgroups of $F$. Assume that $f\in F[X_1, \dots, X_n]$ is a polynomial whose degree satisfies
\begin{align*}
\deg(f)\leq (|A_1|-1)+\dots+(|A_n|-1)+\min\{ |A_1|,\dots,  |A_n|\}-1.
\end{align*}
Then the coefficient of the monomial $X_1^{|A_1|-1}\dots X_n^{|A_n|-1}$ in $f$ equals the averaged sum
\begin{align*}
\frac{1}{|A_1\times \dots \times A_n|}\sum_{(a_1,\dots,a_n)\in A_1\times \dots \times A_n}  a_1\dots a_n f(a_1,\dots,a_n).
\end{align*}
\end{cor}

\begin{cor}[Additive grids]\label{cor: mean0add} Let $F$ have characteristic $p$, and let $A_1,\dots,A_n$ be finite additive subgroups of $F$. Assume that $f\in F[X_1, \dots, X_n]$ is a polynomial whose degree satisfies
\begin{align*}
\deg(f)\leq (|A_1|-1)+\dots+(|A_n|-1)+(1-p^{-1})\min\{ |A_1|,\dots,  |A_n|\}-1.
\end{align*}
If the coefficient of the monomial $X_1^{|A_1|-1}\dots X_n^{|A_n|-1}$ in $f$ is zero, then
\begin{align*}
\sum_{a\in A_1\times \dots \times A_n}  f(a)=0.
\end{align*}
\end{cor}

We note that Corollary~\ref{cor: mean0add} strengthens Corollary~\ref{cor: War2add}. We also point out that the above corollaries can be easily modified to allow for cosets. The resulting weighted sums would be slightly more involved; for the sake of simplicity, we just allude to such a generalized version instead of spelling it out.

Over finite fields, Corollary~\ref{cor: mean0add} yields results in the spirit of Warning's theorem. As a simple illustration, we discuss a variation on Example~\ref{ex: PP}.

\begin{ex}
Let $A_1,\dots,A_n$ be additive subgroups of a finite field $\F_q$ of characteristic $p$. The feature of the grid $A_1\times\dots \times A_n\subseteq \F_q^n$ that we are now interested in is 
\begin{itemize}
\item[(PP$_p$)] a plane in $\F_q^n$ intersects the grid in a number of points, which is a multiple of $p$.
\end{itemize}
Clearly, property (PP$_p$) is stronger than property (PP), considered in Example~\ref{ex: PP}; but the grid is also assumed to be more structured.

Consider again the polynomial $f_\pi\in \F_q[X_1,\dots,X_n]$ of degree $q-1$ associated to a plane $\pi\subseteq \F_q^n$. We wish to apply Corollary~\ref{cor: mean0add} to $f_\pi$. To ensure that the monomial $X_1^{|A_1|-1}\dots X_n^{|A_n|-1}$ does not appear in $f_\pi$, we simply require that $(|A_1|-1)+\dots+(|A_n|-1)\neq q-1$. As $f_\pi$ takes the value $1$ on $\pi$, respectively the value $0$ off $\pi$, we obtain
\begin{align*}
\sum_{a\in A_1\times \dots \times A_n}  f_\pi (a)=N_\pi\cdot 1
\end{align*}
where $N_\pi$ denotes the size of the intersection of $\pi$ with the grid. We deduce that $N_\pi\equiv 0$ mod $p$ for each plane $\pi$, meaning that (PP$_p$) holds, provided that
\[(|A_1|-1)+\dots+(|A_n|-1)\neq q-1,\]
and
\[q-1< (|A_1|-1)+\dots+(|A_n|-1)+(1-p^{-1})\min\{ |A_1|,\dots,  |A_n|\}.\]
\end{ex}

The coefficient theorem~\ref{Thm: CT} is designed to determine the coefficient of a top-degree monomial. The complete coefficient theorem~\ref{Thm: CCT} reaches into the lower-degree monomials of a polynomial. In fact, \emph{all} coefficients could be uncovered by an evaluation over a suitable structured grid. The `complete' designation of Theorem~\ref{Thm: CCT} owes largely to this fact, but also to the crucial use of complete symmetric polynomials in its proof.

We illustrate this `complete' viewpoint by the following multivariable interpolation theorem.

\begin{thm}\label{thm: MIT}  Let $A_1,\dots,A_n$ be $\lambda$-null finite subsets of a field $F$, where $|A_i|>1$ for each $i$. If $f\in F[X_1, \dots, X_n]$ is a polynomial of degree at most $\lambda$, then 
\begin{align*}
f=\sum_{k_1+\dots+k_n\leq \lambda} \bigg(\sum_{a\in A_1\times \dots \times A_n} a_1^{|A_1|-k_1-1}\dots a_n^{|A_n|-k_n-1}\: w_a \:f(a)\bigg) X_1^{k_1}\dots X_n^{k_n}.
\end{align*}
\end{thm} 

\begin{proof}
Let $(k_1,\dots,k_n)$ be a tuple of non-negative integers satisfying $k_1+\dots+k_n\leq \lambda$; in particular, $k_i\leq \lambda\leq |A_i|-1$ for each $i$. Let $C$ denote the coefficient of the monomial $X_1^{k_1}\dots X_n^{k_n}$ in $f$. We have to show that
\[C=\sum_{a\in A_1\times \dots \times A_n} a_1^{|A_1|-k_1-1}\dots a_n^{|A_n|-k_n-1}\: w_a \:f(a).\]
We employ the degree-raising trick. Put 
\[\tilde{f}=X_1^{|A_1|-k_1-1}\dots X_n^{|A_n|-k_n-1} f.\]
Then $C$ is the coefficient of the monomial $X_1^{|A_1|-1}\dots X_n^{|A_n|-1}$ in $\tilde{f}$, and
\begin{align*}
\deg(\tilde{f})&= (|A_1|-k_1-1)+\dots+(|A_n|-k_n-1)+\deg(f)\\
&\leq (|A_1|-1)+\dots+(|A_n|-1)+\lambda.
\end{align*}
Hence, by Theorem~\ref{Thm: CT}, we have that
\begin{align*}
C=\sum_{a\in A_1\times \dots \times A_n} w_a\: \tilde{f}(a).
\end{align*}
Upon reverting to the original polynomial $f$, we obtain the desired formula for $C$.
\end{proof}

\begin{cor}
Let $A_1,\dots,A_n$ be $\lambda$-null finite subsets of a field $F$, where $|A_i|>1$ for each $i$. Let $F\subseteq E$ be a field extension. Assume that $f\in E[X_1, \dots, X_n]$ is a polynomial of degree at most $\lambda$, with the property that the values of $f$ over the grid $A_1\times \dots \times A_n$ lie in $F$. Then $f$ has coefficients in $F$. 
\end{cor}

\section*{Appendix: Sylvester's identity}
This author vaguely recalls having had to check, at some point in middle school, the following identities:
\begin{align*}
\frac{1}{(a-b)(a-c)}+\frac{1}{(b-a)(b-c)}+\frac{1}{(c-a)(c-b)}&=0,\\
\frac{a}{(a-b)(a-c)}+\frac{b}{(b-a)(b-c)}+\frac{c}{(c-a)(c-b)}&=0,\\
\frac{a^2}{(a-b)(a-c)}+\frac{b^2}{(b-a)(b-c)}+\frac{c^2}{(c-a)(c-b)}&=1,
\end{align*}
valid for any three distinct numbers $a, b,c$. Checking similar identities for four distinct numbers $a,b,c,d$ may have happened, but if it did, then it must be a repressed memory. The general result reads as follows.

\begin{thm}[Euler's identities]\label{thm: eul}
Let $a_1,\dots,a_m$ be distinct elements of a field $F$. Then
\begin{align}
\sum_{k=1}^m \frac{a^{d}_k}{\prod_{j: \: j\neq k\:} (a_k-a_j)}=\begin{cases}
0 & \textrm{ if } d=0,\dots,m-2, \\
1 & \textrm{ if } d=m-1.
\end{cases}\label{eq: eul}
\end{align}
\end{thm}

These identities first appeared in a letter of Leonhard Euler to Christian Goldbach, dated September 25th 1762 \cite[p.1123-1124]{E}. Euler goes on to say that these identities appear ``to be more than a little curious; however, it seems to me that you may have been so kind as to communicate something similar to me a long time ago.'' In a subsequent letter, dated November 9th 1762, Euler proves his identities by using a partial fraction decomposition.

But the much more general result is the following.

\begin{thm}[Sylvester's identity]\label{thm: syl}
Let $a_1,\dots,a_m$ be distinct elements of a field $F$. Then, for any non-negative integer $d$, we have
\begin{align}\label{eq: syl}
\sum_{k=1}^m \frac{a^{d}_k}{\prod_{j: \: j\neq k\:} (a_k-a_j)}=\sum_{\substack{i_1+\dots+i_m=d-m+1\\ i_1,\dots,i_m\geq 0}} a_1^{i_1}\dots a_m^{i_m}.
\end{align} 
\end{thm}

One reason why the identity \eqref{eq: syl} is a beautiful and unexpected formula, is that the left-hand side is a rational function in $a_1,\dots,a_m$ while the right-hand side is a polynomial function in $a_1,\dots,a_m$. 

Note how Euler's identities \eqref{eq: eul} are accounted for by Sylvester's identity in low degrees: when $d<m-1$, the right-hand side of \eqref{eq: syl} is an empty sum; when $d=m-1$, the right-hand side of \eqref{eq: syl} has a single term, equal to $1$. Consider now high degrees. When $d\geq m$, the right-hand side of \eqref{eq: syl} can be expressed concisely, by means of the complete symmetric polynomials, as $h_{d-m+1}(a_1,\dots,a_m)$. Setting $A=\{a_1,\dots,a_m\}$ and $\Pi_A(X)=(X-a_1)\dots(X-a_m)$, we see that we can rewrite Sylvester's identity in the concise form of Lemma~\ref{lem: Syl}. We also note that Lemma~\ref{lem: Syl} has a minor technical advantage: it is valid for $m=1$ as well, whereas Theorem~\ref{thm: syl} implicitly assumes that $m\geq 2$.

The attribution of the identity \eqref{eq: syl} to James Joseph Sylvester is explained by Gaurav Bhatnagar in \cite{Bhat}. It seems that the identity \eqref{eq: syl}, and the Sylvester attribution, are not quite absorbed in the mathematical canon. On the one hand, see Knuth \cite[p.472-473, Notes to Exer.33]{Knu}. On the other hand, Sylvester's identity was rediscovered by James Louck, cf. \cite[Appendix A]{LB} and \cite[Thm.2.2]{CL}, and, more recently, by Volker Stehl and Herbert Wilf \cite[Sec.4]{SW}. As a further illustration of this point, that Sylvester's identity is relatively unknown, consider the following result from a very recent and very interesting work of Stephan Ramon Garcia, Mohamed Omar, Christopher O'Neill, and Samuel Yih \cite[Thm.3]{G++}: for distinct $a_1\dots,a_m\in \C$, and $z\in \C$, it holds that
\begin{align*}
\sum_{p=0}^\infty \frac{h_p(a_1\dots,a_m)}{(p+m-1)!}\: z^{p+m-1}=\sum_{k=1}^m \frac{e^{a_k z}}{\prod_{j: \: j\neq k\:} (a_k-a_j)}.
\end{align*} 
This is a key identity for the purposes of \cite{G++}. The proof given therein is quite lengthy, occupying well over three pages. The following three-line argument shows that the above identity is, in fact, a restatement of \eqref{eq: syl} as a power series identity:
\begin{align*}
\sum_{k=1}^m \frac{e^{a_k z}}{\prod_{j: \: j\neq k\:} (a_k-a_j)}&=\sum_{k=1}^m \frac{1}{\prod_{j: \: j\neq k\:} (a_k-a_j)} \bigg(\sum_{d=0}^\infty \frac{(a_kz)^d}{d!}\bigg)\\
&=\sum_{d=0}^\infty \bigg(\sum_{k=1}^m \frac{a_k^d}{\prod_{j: \: j\neq k\:} (a_k-a_j)}\bigg)\: \frac{z^d}{d!}\\
&=\sum_{d=m-1}^\infty h_{d-m+1}(a_1\dots,a_m)\: \frac{z^d}{d!}.
\end{align*} 

Both Stehl - Wilf \cite{SW} and Chen - Louck \cite{CL} prove Theorem~\ref{thm: syl} by means of generating functions. Another proof can be found in \cite[Sec.2]{Bhat}. Here, we include a seemingly new proof of Theorem~\ref{thm: syl}. This is an argument that Euler himself could have given. 

\begin{proof}
We start, very much like Euler did for his proof of Theorem~\ref{thm: eul}, with a partial fraction decomposition:
\begin{align}\label{eq: pfd}
\frac{1}{(X-a_1)\dots(X-a_{m-1})}=\sum_{k=1}^{m-1}\Bigg( \prod_{\substack{j=1\\j\neq k}}^{m-1}\frac{1}{a_k-a_j}\Bigg) \frac{1}{X-a_k}.
\end{align}
Evaluating at the remaining node, $X=a_m$, and multiplying through by $a^d_m$ yields
\begin{align*}
\Bigg( \prod_{\substack{j=1\\j\neq m}}^{m}\frac{1}{a_m-a_j}\Bigg)a_m^d=-\sum_{k=1}^{m-1}\Bigg( \prod_{\substack{j=1\\j\neq k}}^{m}\frac{1}{a_k-a_j}\Bigg) a_m^d
\end{align*}
and so, adding terms on both sides, 
\begin{align*}
\sum_{k=1}^{m}\Bigg( \prod_{\substack{j=1\\j\neq k}}^{m}\frac{1}{a_k-a_j}\Bigg) a_k^d=\sum_{k=1}^{m-1}\Bigg( \prod_{\substack{j=1\\j\neq k}}^{m}\frac{1}{a_k-a_j}\Bigg) (a_k^d-a_m^d).
\end{align*}
The left-hand side of the above identity is the left-hand side of \eqref{eq: syl}, which we denote by $S_d(a_1,\dots,a_m)$ in what follows. The right-hand side of the above formula can be rewritten, after dividing $a_k^d-a_m^d$ by $a_k-a_m$, as 
\begin{align*}
\sum_{k=1}^{m-1}\Bigg( \prod_{\substack{j=1\\j\neq k}}^{m-1}\frac{1}{a_k-a_j}\Bigg) \sum_{e=0}^{d-1} a_k^e a_m^{d-e-1}=\sum_{e=0}^{d-1}a_m^{d-e-1}\sum_{k=1}^{m-1}\Bigg( \prod_{\substack{j=1\\j\neq k}}^{m-1}\frac{1}{a_k-a_j}\Bigg)  a_k^e. 
\end{align*}
We recognize the inner sum as $S_e(a_1,\dots,a_{m-1})$. All in all, we get the recurrence 
\begin{align*}
S_d(a_1,\dots, a_{m})=\sum_{e=0}^{d-1} S_e(a_1,\dots,a_{m-1}) \: a_{m}^{d-e-1}.
\end{align*}
Now let us turn to the right-hand side of \eqref{eq: syl}. We denote by $H_d(a_1,\dots,a_m)$, and we observe that it satisfies the same recurrence relation. Indeed:
\begin{align*}
H_d(a_1,\dots,a_m)&=\sum_{\substack{i_1+\dots+i_m=d-m+1}} a_1^{i_1}\dots a_m^{i_m}\\
&=\sum_{e=m-2}^{d-1} \Bigg(\sum_{i_1+\dots+i_{m-1}=e-m+2} a_1^{i_1}\dots a_{m-1}^{i_{m-1}}\Bigg) a_{m}^{d-e-1} \\
&=\sum_{e=0}^{d-1} H_e(a_1,\dots,a_{m-1}) \: a_{m}^{d-e-1}.
\end{align*}
Throughout, the exponents $i_1,\dots,i_m$ are non-negative. In the last step, we have extended the range of $e$, but note that each one of the corresponding terms vanishes. 

The desired identity $S_d(a_1,\dots,a_m)=H_d(a_1,\dots,a_m)$ will follow as soon the boundary cases $d=0$, respectively $m=2$, are checked. Let us do that. When $d=0$, we actually need to check that $S_0(a_1,\dots,a_m)=0$; this fact already appeared on the way, for it is the evaluation of \eqref{eq: pfd} at $X=a_m$. When $m=2$, we need to check that
\[\frac{a^{d}_1}{a_1-a_2}+\frac{a^{d}_2}{a_2-a_1}=\sum_{\substack{i_1+i_2=d-1\\ i_1,i_2\geq 0}} a_1^{i_1} a_2^{i_2}.\]
This is obvious.
\end{proof}


\end{document}